\documentclass[12pt]{article}
\usepackage{latexsym,color,amsmath,amsthm,amssymb,amscd,amsfonts}
\usepackage{scalefnt}
\usepackage[font=small,labelfont=bf]{caption}
\usepackage{float}
\usepackage{graphicx}
\usepackage{amsopn}
\usepackage{relsize}
\usepackage{mathrsfs}
\usepackage{upgreek}

\newtheoremstyle{nonum}{}{}{\itshape}{}{\bfseries}{.}{ }{\thmnote{#3}}

\newtheorem{thm}{Theorem}[section]
\newtheorem*{thm*}{Theorem}

\newtheorem{lem}[thm]{Lemma}
\newtheorem{prop}[thm]{Proposition}
\newtheorem{rem}[thm]{Remark}

\newtheorem*{definition*}{Definition}

\newtheorem*{rems*}{Remarks}

\theoremstyle{nonum}

\newcommand{\R}{\mathbb R}
\newcommand{\RR}{\mathbb R}

\def\A{{\cal A}}
\def\L{{\cal L}}

\def\J{{\cal J}}

\newcommand{\iprod}[2]{\langle #1,#2 \rangle} 

\def\calA{{\cal A}}
\def\calL{{\cal L}}
\def\L{{\calL}}

\def\measure{\mu}

\def\vol{{\rm Vol}}

\def\epi{{\rm epi}\,}

\def\cvx{{\rm Cvx}}

\def\eps{{\varepsilon}}

\def\conv{{\rm conv}}

\newcommand{\infc}{\Box}                         
\newcommand{\jinfc}{\boxdot}                     

\begin{document}
\title{Functional Brunn-Minkowski Inequalities\\ Induced by Polarity}
\date{}
\author{S. Artstein-Avidan, D.I. Florentin, A. Segal}
\maketitle
\begin{abstract}

We prove a new family of inequalities, which compare the integral of a geometric
convolution of non-negative functions with the integrals of the original
functions. For classical inf-convolution, this type of inequality is called the
Pr\'{e}kopa-Leindler inequality, which, restricted to indicators of convex bodies,
gives the classical Brunn-Minkowski inequality. The convolution we consider
is a dif{}ferent one, which arises from the study of the polarity transform for
functions. While inf-convolution arises as the pull back of usual
addition of convex functions under the Legendre transform, our geometric inf-
convolution arises as the pull back of the second order reversing transform on
geometric convex functions (called either polarity transform or ${\calA}$-transform).
These are, up to linear terms, the only order reversing isomorphisms on the
class of geometric convex functions. We prove that the integral of this new
geometric convolution of two functions is bounded from below by the harmonic
average of the individual integrals. Our inequality implies the Brunn-Minkowski
inequality, as well as some other, new, inequalities for volumes of bodies. Our
inequalities are intimately connected with Busemann's convexity theorem, a
new variant of which we prove for 1-convex hulls and log-concave densities.
\end{abstract}

\section{Introduction and results}\label{Sec_Intro}

The Brunn-Minkowski inequality is a fundamental inequality in geometry which has numerous applications in mathematics and beyond. In one of its many equivalent forms, it states that the volume of the Minkowski average of two sets in $\RR^n$, to the power $\frac{1}{n}$ is bounded from below by the average of the volumes of the individual sets to the same power 
(see, e.g. \cite[Chapter 1]{AGM}). The impact of this inequality extends far beyond geometry and analysis, see e.g. \cite{Gardner}. 

In the literature, it is common to view the Pr\'{e}kopa-Leindler inequality \cite{Leindler, Prekopa} as ``the functional analogue of the Brunn-Minkowski inequality''. It compares the integral of some weighted average of two functions $f$ and $g$, with the geometric average of their integrals. More precisely, 
let $f,g:\R^n\to \R^+$ and $\lambda\in
(0,1)$ and consider the function $h_\lambda$ given by 
\begin{equation}\label{cond-PL} h_\lambda(z) = \sup \left\{f(x)^{1-\lambda}g(y)^{\lambda}  :  z= (1-\lambda)x+ \lambda y\right\}.  \end{equation}
This function is called the $\lambda$-sup-convolution of $f$ and $g$, and the Pr\'{e}kopa-Leindler theorem states that 
\begin{equation}\label{eq:PL}
\int  h_\lambda \ge
\left(\int  { f} \right)^{1-\lambda}
\left(\int  { g} \right)^\lambda. 
\end{equation}
In case the function $h_\lambda$ is not measurable, the integral should be understood as an upper integral.
The special case of $f = 1_B$ and $g = 1_A$, indicators of convex sets in $\RR^n$, yields the Brunn-Minkowski inequality, as one easily checks that $h_\lambda = 1_{(1-\lambda)A+\lambda B}$. 

In this paper we claim that alongside the Pr\'{e}kopa-Leindler inequality road to functionalization, which is governed, as we explain below, by the classical Legendre transform, there resides a second, less travelled road towards functionalization of the Brunn-Minkowski inequality, governed by a ``cousin'' of the Legendre transform, which is called the ``polarity transform''. Our main theorem reads as follows:   

\begin{thm}\label{Thm-ginf-measurable}
	Let $f,g:\R^n\to [0,1]$ be  measurable functions, and $\lambda\in (0,1)$. Then the function
	\[
	h_\lambda(z)  =
	\sup_{0< t< 1}\,
	\sup_{z = (1-t)x+ty}
	\left\{ \min
	\left\{
	f(x)^{\frac{1 - t}{1 - \lambda}},
	g(y)^{\frac{  t  }{    \lambda}}
	\right\}\right\}
	\]
	satisfies the following inequality
	\[
	\int_{\R^n}h_\lambda \ge
	\left(
	(1 - \lambda)  \left(\int_{\R^n}f \right)^{-1} +
	\lambda   \left(\int_{\R^n}g \right)^{-1}
	\right)^{-1}.
	\]
	The integral of the function $h_\lambda$ (which is not necessarily measurable), should be interpreted as an upper integral.
\end{thm}

Let us provide some geometric background for the classical Pr\'{e}kopa-Leindler inequality and the new inequality.
In the case where the functions are log-concave, namely when $f = e^{-\varphi}$ and $g = e^{-\psi}$ for convex $\varphi$ and $\psi$, the sup-convolution average of the two functions, $h_\lambda$, is log-concave again. In this case it can be equivalently defined using the  Legendre transform $\cal L$ given for  a function $\phi$ by
\[
(\L \phi)(y) =\sup_{x}\{\iprod{x}{y} - \phi(x)\}, 
\]
and one has
 $h_\lambda = e^{-\varphi\square_\lambda \psi}$ where
\begin{equation}\label{eq:infconvolution-and-L}
\L\left(\varphi \square_\lambda \psi\right) = (1-\lambda)\L \varphi + \lambda \L\psi. 
\end{equation}
In other words, the linear structure of inf-convolution is simply the
pull-back of the standard linear structure of pointwise addition for
functions, under the Legendre transform. Rephrasing yet again, on the
class of convex functions, the Pr\'{e}kopa-Leindler theorem states
that the functional $\phi\mapsto\int e^{-\phi}$ is log-concave with
respect to inf-convolution, i.e.
\begin{equation}\label{eq:PL-2f}
\int       e^{-\varphi\square_\lambda \psi} \ge
\left(\int e^{-\varphi} \right)^{1-\lambda}
\left(\int e^{ -\psi }  \right)^\lambda.
\end{equation}

The Legendre transform, which is in some sense the ``generator'' of
the inf-convolution operation and thus of the Pr\'{e}kopa-Leindler
inequality, is a very well known and useful operation on convex
functions. In \cite{AM2} it was shown that (up to linear terms) it is
the unique order reversing bijection on convex lower semi continuous
functions. However, in \cite{AM4} it was shown that on the subclass
$\cvx_0(\R^n)$ of {\em geometric convex functions}, namely
non-negative lower semi continuous convex functions vanishing at the
origin, there exist precisely two, essentially dif{}ferent order
reversing bijections. One is the Legendre transform, and the other is
the so called {\em polarity transform} $\A$, defined by
\begin{equation}\label{eq:def-of-A}
({\A}\phi)(y) =
\left\{
\begin{array}{lll}
\sup_{\{x\in\R^n: \phi(x)>0\}}
\frac{\iprod{x}{y}-1}{\phi(x)} &
\mbox{if}~~ 0\neq y \in \{\phi^{-1}(0)\}^{\circ} \\
0 & \mbox{if}~~ y=0 \\
{+\infty} & \mbox{if}~~ y \not\in \{\phi^{-1}(0)\}^{\circ}
\end{array}
\right\},
\end{equation}
with the convention $\sup\emptyset = 0$. The transform $\A$ appears
in \cite{Rockafellar} and a similar transform appears in
\cite{Milman-before-Rock}, but it seems to have been left virtually
untouched until rediscovered recently, and proved to be the only
additional order reversing isomorphism on $\cvx_0(\R^n)$. For
properties of $\A$ and dif{}ferent geometric interpretations, see
\cite{AM4}. For dif{}ferential analysis concerning $\A$, and
applications to solving families of dif{}ferential equations, see
\cite{AR}. We mention in this context that our main theorems
can be interpreted as volume estimates for solutions of certain partial
dif{}ferential equations (those linearized by $\A$, as portrayed in
\cite{AR}) in terms of the boundary or initial conditions.

By pulling back the standard linear structure (of pointwise addition
of functions) under the polarity transform $\A$, we obtain a new,
completely dif{}ferent linear structure on $\cvx_0(\R^n)$. In this
linear structure, the sum  of two functions $\varphi$ and $\psi$ is
defined by $\A \left(\A (\varphi) + \A (\psi)\right)$, and
multiplication by a constant is defined similarly. Since the focus of
this paper is the study of concavity properties of certain
functionals, we are interested in {\em averaging} with respect to
this linear structure, thus we introduce the geometric infimum
convolution. Given $\lambda\in[0,1]$ and geometric convex functions
$\varphi,\psi \in \cvx_0(\R^n)$, their {\em geometric
$\lambda$-inf-convolution} is defined by
\[
\varphi \boxdot_\lambda \psi :=
\A
\left( 
	(1 - \lambda) \A\varphi + \lambda \A\psi
\right).
\]

In Section \ref{sec:ginf-formulae} we shall show the validity of the
following formula for the {{geometric $\lambda$-inf-convolution}} of
two geometric convex functions:
\begin{equation}\label{eq:newjinfform}
(\varphi \boxdot_\lambda \psi) (z) =
\inf_{0< t< 1}\, \inf_{z = (1-t)x+ty}
\left\{
\max
\left\{
\frac{1-t}{1-\lambda} \varphi(x), \frac{t}{\lambda}  \psi(y)
\right\}
\right\}.
\end{equation}
We use this equivalent definition of the geometric inf-convolution to
extend it for any two non-negative functions (not necessarily
convex). Considering the functions $f=e^{-\varphi},\, g=e^{-\psi},\,
h_\lambda=e^{-\varphi \boxdot_\lambda \psi}$ from $\R^n$ to $[0,1]$,
we obtain the geometric sup-convolution:
\[
h_\lambda(z)  =
\sup_{0< t< 1}\,
\sup_{z = (1-t)x+ty}
\left\{ \min
\left\{
f(x)^{\frac{1 - t}{1 - \lambda}},
g(y)^{\frac{  t  }{    \lambda}}
\right\}\right\}.
\]

This provides a completely new linear structure on the cone of
non-negative functions, which is dominated by the second ``duality''
for geometric convex functions. This paper is focused on establishing
certain concavity results, the first of which is Theorem
\ref{Thm-ginf-measurable} which we quoted above, which essentially
states that the functional $\varphi \to \int e^{-\varphi}$ is
$(-1)$-concave with respect to the new linear structure. More
precisely, we show that
\begin{equation}\label{eq:PLjinf-2f}
\int_{\R^n} e^{-\varphi \boxdot_\lambda \psi}\ge
\left(
(1 - \lambda) \left(\int_{\R^n}e^{-\varphi} \right)^{-1} +
\lambda  \left(\int_{\R^n}e^{-\psi} \right)^{-1}
\right)^{-1},
\end{equation}
whenever $\lambda\in(0,1)$, $\varphi, \psi:\R^n\to [0,\infty]$. 
A geometric interpretation of \eqref{eq:newjinfform} will also be
given in Section \ref{sec:ginf-formulae}. Note that the geometric
average of the Pr\'ekopa-Leindler inequality is replaced here with a
harmonic average.

We mention that our new inequality implies the Brunn-Minkowski
inequality since, as one may easily check, the functions
$f(x) = e^{-\|x\|_K}$, $g(y) = e^{-\|y\|_T}$ and
$h_\lambda(z) = e^{-\|z\|_{(1-\lambda)K + \lambda T}}$ satisfy the
inequality in the statement of the theorem, so we get that
\[
\vol((1-\lambda)K + \lambda T)  \ge \left(
(1 - \lambda)  \left(\vol( K ) \right)^{-1} +
\lambda   \left(\vol(   T) \right)^{-1}
\right)^{-1},
\]
which by homogeneity implies the classical Brunn-Minkowski inequality
(see, e.g. \cite[Chapter 1]{AGM}). In this sense, our new inequality
may be considered as a functional version of the Brunn-Minkowski
inequality, in a completely dif{}ferent direction than that of the
classical Pr\'ekopa-Leindler inequality. In fact, our generalization
of Theorem \ref{Thm-ginf-measurable} implies a generalized version of
the Brunn-Minkowski inequality as a particular case, see Remark
\ref{rem:new-ineq-4-bodies}.

In the second part of the paper we extend Theorem \ref{Thm-ginf-measurable}
to integration with respect to a general log-concave measure. Note
that while in the classical Pr\'ekopa-Leindler theory, multiplying
$f,g$ and $h$ by a log-concave density yields three functions which
again satisfy the pointwise equality \eqref{cond-PL}, here this is no
longer the case, and one must provide an independent proof.

To do this we utilize the fact that the geometric $\lambda$-inf-convolution operation is
intimately related to the notion of intersection bodies, an operation defined and used in the classical Busemann Theorem. We discuss this relation in detail in Section \ref{Sec_Buse}, and use methods from known Busemann-type theorems to
prove the following theorem.

\begin{thm}\label{Thm_ginf-mu}
Let $\measure$ be a log-concave measure on $\R^n$,
and let $f,g:\R^n\to [0,1]$ be measurable functions which have a finite integral with respect to $\measure$. For any $\lambda\in (0,1)$, the function
\[
h_\lambda(z)  =
\sup_{0< t< 1}\,
\sup_{z = (1-t)x+ty}
\left\{ \min
\left\{
f(x)^{\frac{1 - t}{1 - \lambda}},
g(y)^{\frac{  t  }{    \lambda}}
\right\}\right\}
\]
satisfies
\begin{equation}\label{Ineq_GeomPL-with-log-conc-measure}
	\int_{\R^n}h_\lambda\, d\measure \ge
	\left(
	(1 - \lambda) \left(\int_{\R^n}f\, d\measure \right)^{-1} +
	\lambda  \left(\int_{\R^n}g\, d\measure \right)^{-1}
	\right)^{-1}.
	\end{equation}
\end{thm}

\begin{rem}\label{rem:new-ineq-4-bodies}
Assume $K,T$ are convex bodies containing the origin (not necessarily
in their interior) and let $||\cdot||_K, ||\cdot||_T$ denote their
corresponding Minkowski functionals. Applying inequality
\eqref{Ineq_GeomPL-with-log-conc-measure} to the triplet
$f(x) = e^{-\|x\|_K}$, $g(y) = e^{-\|y\|_T}$ and
$h(z) = e^{-\|z\|_{(1-\lambda)K + \lambda T}}$ (which satisfy the
assumption of Theorem \ref{Thm_ginf-mu}), we get that letting
\[
A_\mu(K) =
\int_{\R^n}\,e^{-\|x\|_K} d\mu (x) =
\int_{0}^\infty e^{-s} \mu(sK)\, ds,
\]
one has 
\[
A_\mu ((1-\lambda)K + \lambda T) \ge
\left(
	(1 - \lambda) \left(A_\mu(K)\right)^{-1} +
		 \lambda  \left(A_\mu(T)\right)^{-1}
\right)^{-1}.
\]

The functional $A_\mu$, defined here for convex bodies, can be extended
to a measure on $\R^n$. If the measure $\mu$ has density $e^{-\phi}$,
then the measure $A_\mu$ has density given by
\[
\left( \frac{dA_\mu}{dx} \right) (x) =
\int_0^\infty s^n e^{-s} e^{-\phi(sx)} ds.
\]
\end{rem}

Our methods turn out to be quite general, and may be used to prove
concavity results for various other ``volume'' functionals on
functions. We illustrate this by proving yet another result of this
type. Denoting for any $p>0$ and measurable $\phi:\R^n\to\R^+$,
$\left\|\phi\right\|_p = \left( \int \phi^p \right)^\frac{1}{p}$
(clearly this is a norm only when $p\ge 1$), we have
\begin{thm}\label{Thm-Lp-norms}
Let $p>0$ and let $f,g,h:\R^n\to \R^+$ be measurable functions such
that $\left\|f\right\|_p$, $\left\|g\right\|_p$, and
$\left\|h\right\|_p$ are finite. Assume that
\[
h((1-t)x + ty) \ge
\min\left\{ \frac{f(x)}{1-t}, \frac{g(y)}{t}\right\},
\]
whenever $t\in (0,1)$ and $x, y\in\R^n$. Then 
\[
\left\| h \right\|_p
\ge
\left\| f
\right\|_p + 
\left\| g \right\|_p.
\]
\end{thm}

\noindent {\bf Acknowledgements:}
The authors would like to thank Bo'az Klartag for pointing out the
relation between Busemann's theorem and the geometric
$\lambda$-inf-convolution operation. We would also like to thank the
anonymous referee for his helpful comments and for pointing out
\cite{Rin}. The first named author was supported in part by ISF grant
No. 665/15. The second named author was supported in part by the U.S.
National Science Foundation Grant DMS-1101636.

\section{Extending the geometric inf-convolution}\label{sec:ginf-formulae}
In this section we extend the operation of the geometric
inf-convolution (which we abbreviate by {\em ginf-convolution}) to
act between any two non-negative functions $\varphi$ and $\psi$. We
then present a new formula for the geometric
$\lambda$-inf-convolution.

In \cite{AM4}, it was shown that on $\cvx_0(\R^n)$ one has $\A\circ\L
=\L\circ\A = \J$, where $\J$ is the order preserving bijection called
{\em the gauge transform} and is given by
\begin{equation}\label{eq:J1}
(\J \phi)(x) = \inf\left\{r>0 : \phi\left(\frac{x}{r}\right)\le \frac{1}{r} \right\}.
\end{equation} 
Thus, given $\varphi,\psi \in \cvx_0(\R^n)$ one has
\begin{eqnarray*}
\J( \varphi \boxdot_\lambda \psi) &=&
\J(\A ((1-\lambda)\A\varphi +\lambda \A \psi)) =
\L((1-\lambda)\A \varphi  + \lambda \A \psi)  \\ &=&
\L((1-\lambda)\L\J \varphi  + \lambda \L\J \psi) =
(\J \varphi ) \square_\lambda  (\J \psi).
\end{eqnarray*}
It is sometimes easier to consider the epi-graphs of functions.  
Denoting the epi-graph of a function $\phi$ by
\begin{equation}\label{def-epi}
\epi(\phi ) = \{(x,z)\in \RR^n \times \RR :\phi (x)< z \},
\end{equation}
one may easily check that the epi-graph of $\varphi\square_\lambda\psi$
is the Minkowski average of the epi-graphs of $\varphi$ and
$\psi$:
\[
\epi(\varphi\square_\lambda\psi) = 
(1-\lambda) \epi(\varphi) + \lambda\epi(\psi). 
\]
Therefore, we get the following relation for the composition $\J$
\begin{equation}\label{Eq_ginf-by-summing-epiJ}
\epi(\J(\varphi \boxdot_\lambda \psi)) =
(1-\lambda) \epi(\J \varphi) + \lambda \epi(\J \psi).
\end{equation}
That is, the ginf-convolution operation  corresponds to the pullback
of Minkowski addition of epi-graphs under the $\J$ transform. In the
same paper \cite{AM4} it was shown that $\J$ is induced by a point
map on $\R^{n}\times \R^+$, given by $F(x,z) = \left(\frac{x}{z},
\frac{1}{z}\right)$. More precisely, for $\phi \in \cvx_0(\R^n)$ we
have that
\begin{equation}\label{Eq_F-induces-J}
F(\epi(\phi)) = \epi(\J \phi). 
\end{equation}
Thus we could equivalently define the ginf-convolution, or the
$\lambda$-ginf-convolution by
\begin{eqnarray*}
\epi (\varphi\boxdot \psi) &=&
F(F(\epi(\varphi)) + F(\epi(\psi))), \\
\epi (\varphi\boxdot_\lambda  \psi) &=&
F((1-\lambda)F(\epi(\varphi)) + \lambda F(\epi(\psi))).
\end{eqnarray*}
This definition may be
extended to the set of non-negative  functions.
\begin{prop}\label{Prop_ginf-for-positive-is-well-def}
Let $\lambda\in(0,1)$ and let $\varphi, \psi : \R^n\to [0,\infty]$.
The sets
\begin{equation*}
	F(F(\epi(\varphi)) + F(\epi(\psi))), \qquad
	F((1-\lambda)F(\epi(\varphi)) + \lambda F(\epi(\psi))),
\end{equation*}
are epi-graphs of functions from $\R^n$ to $[0,\infty]$. We denote
these functions by $\varphi\boxdot \psi$ and $\varphi\boxdot_\lambda
\psi$ respectively.
\end{prop}
\begin{proof}
Recall that a set $S \subset \R^n$ is called {\em star shaped}, if
for any $x \in S$ the line segment $[0,x]$ connecting $x$ to the
origin is contained in $S$. Note that $F$ is an involution on
$\R^n \times \R^+$ which maps vertical fibers to intervals with one
endpoint at the origin. More precisely, for $x\in\R^n$ and $c>0$ we
have $F(\{x\}\times[c,\infty))= \frac{1}{c} \cdot (0,\, (x,1)]$. Thus
$F$ maps epi-graphs to star shaped sets in $\R^n \times \R^+$, and
such star shaped sets are mapped to epi-graphs. Since the class of
star shaped sets is closed under Minkowski addition, the statement
follows.
\end{proof}
Next, we turn to proving formula \eqref{eq:newjinfform}.
\begin{prop}\label{Prop_new-formula}
Let $\lambda\in(0,1)$ and let $\varphi, \psi : \R^n\to [0,\infty]$.
Then we have
\begin{equation}\label{eq:jinf-form} 
(\varphi\boxdot \psi) (z) = 
\inf_{0< t< 1}\, \inf_{z = (1-t)x + ty}
\max \left\{(1-t)\varphi(x), t\psi(y)\right\},
\end{equation}
and
\begin{equation*}
(\varphi \boxdot_\lambda \psi) (z) =
\inf_{0< t< 1}\, \inf_{z = (1-t)x+ty}
\max \left\{\frac{1-t}{1-\lambda} \varphi(x), \frac{t}{\lambda}  \psi(y)\right\}.
\end{equation*}
\end{prop}

\begin{rem}{\rm
We mention that the usefulness of such a formula goes beyond its
usage in this paper. Indeed, for the case of $\varphi, \psi \in
\cvx_0(\R^n)$ we get by means of this formula an expression for the
polar of a sum:
\[
\A(\varphi + \psi)(y) =
(\A\varphi \boxdot \A\psi)(y) =
\inf_{0< t< 1}\, \inf_{y = (1-t)y_1 + ty_2}
\max \left\{(1-t)\A\varphi(y_1), t\A\psi(y_2)\right\}.
\]
Here one clearly sees that the domain of $\A(\varphi + \psi)(y)$ is
the convex hull of the domains of $\A\varphi$ and $\A\psi$. In fact,
one may figure out where the infimum is attained (say, in the
dif{}ferentiable case) as was shown in \cite[Lemma 8.2]{AR}.}
\end{rem}

\begin{proof}[{\bf Proof of Proposition \ref{Prop_new-formula}}]
Note that since $\lambda F(\epi (\phi)) = F(\epi (\phi/\lambda))$, it
follows
\begin{equation}\label{Eq_lambda-ginf-vs-ginf}
\varphi\boxdot_\lambda  \psi =
\frac{\varphi}{(1-\lambda)}\boxdot \frac{\psi}{\lambda},
\end{equation}
so it suf{}fices to prove \eqref{eq:jinf-form}. We have
\begin{eqnarray*}
\epi (\varphi\boxdot \psi) &=&
F\left(F(\epi (\varphi)) +  F(\epi (\psi))\right) =\qquad \\
&=&F\left(\left\{
\left(\frac{x}{s}+\frac{y}{t}, \frac{1}{s}+\frac{1}{t}\right):
\varphi(x)<s, \psi(y)<t
\right\}
\right)\\
&=&\left\{ \left(  \frac{\frac{x}{s}+\frac{y}{t}}{\frac{1}{s}+\frac{1}{t}},\frac{1}{\frac{1}{s}+\frac{1}{t}} 
 \right) : \varphi(x)<s, \psi(y)<t\right\}.	
\end{eqnarray*}
Therefore, 
\[
(\varphi\boxdot \psi) (z) =
\inf\left\{
\frac{1}{\frac{1}{s}+\frac{1}{t}}  :
z = \frac{\frac{x}{s}+\frac{y}{t}}{\frac{1}{s}+\frac{1}{t}},
\varphi(x)<s, \psi(y)<t
\right\}.
\]
Rewriting we get 
\[
(\varphi\boxdot \psi) (z) =
\inf\left\{
\frac{st}{{s}+{t}}  : z = (1-a)x + ay,
a = \frac{s}{s+t}, \varphi(x)<s, \psi(y)<t
\right\}.
\]
We take the infimum in two steps, first over all choices of $x,s$ and
$y,t$ which satisfy the conditions for a fixed $a$, and then over all
$a\in (0,1)$. We claim that for any fixed $a\in (0,1)$ we have
\begin{eqnarray}\label{eq:oneinfisanotherinf}
&\inf&\left\{
\frac{st}{{s}+{t}}  : z = (1-a)x + ay,
a = \frac{s}{s+t}, \varphi(x)<s, \psi(y)<t
\right\} \\  &=&
\inf \left\{
\max \{(1-a)\varphi(x), a\psi(y)\}:
z = (1-a) x + a y
\right\}.\nonumber
\end{eqnarray}
Indeed we have $\frac{st}{s+t} = (1-a)s > (1-a)\varphi(x)$ and
$\frac{st}{s+t} = at > a\psi(y)$ for any $x,s,y,t$ participating in
the first infimum, thus the left hand side is not smaller than the
right hand side.

For the other direction, assume that for a given $x, y$ which satisfy
$z = (1-a)x + ay$ we have $(1-a)\varphi(x) \ge a\psi(y)$. Let us take
$s = \varphi(x)+\eps$ and choose $t$ so that $a = \frac{s}{s+t}$,
that is, $at = (1 - a)s = (1 - a)(\varphi(x) + \eps) \ge
a\psi(y) + (1-a)\eps > a\psi(y)$. Since $t>\psi(y)$, we get that
$t$ participates in the infimum of the left hand side, and we have 
\[
\frac{st}{{s}+{t}} =
\frac{\varphi(x)t + \eps t}{s + t} <
\frac{\varphi(x)t}{s + t} + \eps =
(1 - a)\varphi(x) + \eps =
\max \{ (1-a)\varphi(x), a\psi(y) \} + \eps.
\] 
Since $\eps$ was arbitrary, we get that fixing $z, a, x, y$
\begin{eqnarray*}
\inf\left\{
 \frac{st}{{s}+{t}}  : z =(1-a)x + ay, a = \frac{s}{s+t},\varphi(x)<s, \psi(y)<t
\right\}
\le
(1-a)\varphi(x),
\end{eqnarray*} 
in the case where $z = (1-a)x + ay$, and $\max \{ (1-a)\varphi(x),
a\psi(y) \} = (1-a)\varphi(x)$. The exact same reasoning works when
the maximum of the two is $a\psi(y)$. Therefore, the two expressions
in \eqref{eq:oneinfisanotherinf} are the same for any fixed $a\in
(0,1)$, and we get that 
\[
(\varphi\boxdot \psi) (z) =
\inf_{0<a<1}
\inf \{ \max \{(1-a)\varphi(x), a\psi(y)\}: {z = (1-a)x + ay} \},
\]
which completes the proof. 
\end{proof}
\begin{rem}{\rm 
For $m$ functions one easily checks the validity of the formulas	
\begin{eqnarray*}
\left(\varphi_1\boxdot \cdots \boxdot \varphi_m\right) (z) &=& 
	\inf_{\sum_{i=1}^m t_i =1}\, \inf_{z = \sum_{i=1}^m t_i x_i}
	\max_{i=1, \ldots, m} \left\{t_i\varphi_i(x_i)\right\},  \\	
\A\left(\varphi_1+ \cdots + \varphi_m\right) (z) &=& 
\inf_{\sum_{i=1}^m t_i =1}\, \inf_{z = \sum_{i=1}^m t_i x_i}
\max_{i=1, \ldots, m} \left\{t_i\A\varphi_i(x_i)\right\}. 
\end{eqnarray*}}
\end{rem}

\section{Pr\'ekopa-Leindler type theorems}
In this section we prove Theorem \ref{Thm-ginf-measurable} and
Theorem \ref{Thm-Lp-norms}. To present these proofs, and also a very
simple proof for the classical Pr\'ekopa-Leindler inequality
\eqref{eq:PL} (which appeared first in \cite{Rin}), we will use
C.~Borell's theorem on concavity of measures. It will be useful to
introduce the following notation for the $p$-average of non-negative
numbers. Given $x,y > 0$, $\lambda \in [0,1]$, and $p\in\R$, denote  
\[
M_p^\lambda(x,y) := ((1-\lambda) x^p + \lambda y^p)^{\frac{1}{p}}.
\]
The cases $p=0, \pm\infty$ are interpreted by limits, so that
$M_0^\lambda(x,y) = x^{1 - \lambda} y^\lambda$, $M_\infty(x,y) = \max
\{x,y\}$ and $M_{-\infty}(x,y) = \min\{x,y\}$. Recall that a measure
$\mu$ on $\R^n$ is called $\kappa$-concave if for all non empty
Borel sets $A, B, C$ such that $(1-\lambda) A + \lambda B \subseteq C$, one has
\begin{equation}\label{Eq_concave-measures}
\measure(C) \ge M_{\kappa}^\lambda(\measure(A), \measure(B)).
\end{equation}
A function $f$ on $\R^n$ is called $\kappa$-concave if for all
$x, y\in\R^n$: 
\[
f((1-\lambda) x + \lambda y) \ge M_{\kappa}^\lambda(f(x), f(y)).
\]
In particular, if $\kappa>0$ then $\kappa$-concavity means that $f^{\kappa}$ is concave whereas if $\kappa<0$ then $\kappa$-concavity means that $f^{\kappa}$ is convex.  When $\kappa = 0$ we call $\kappa$-concavity {\em log-concavity}.

In \cite{Borell}, Borell proved the following classical result
connecting the concavity of a measure with the concavity of its
density function:
\begin{thm}[Borell] \label{thm-Borell}
Let $\mu$ be an absolutely continuous  measure on $\R^n$ with density
$f$, and $n$-dimensional support set. Assume that $\kappa_f\in
[-\frac{1}{n},\infty],\,\kappa_\mu\in [-\infty, \frac{1}{n}]$ satisfy
\[
\kappa_f   = \frac{\kappa_\mu}{1-n\kappa_\mu},\qquad\qquad
\kappa_\mu = \frac{\kappa_f  }{1+n\kappa_f}.
\]
Then $f$ is $\kappa_f$-concave if and only if $\measure$ is
$\kappa_\mu$-concave.
\end{thm}

In particular, Borell's result implies that a measure  is log-concave
if and only if its density is  log-concave. Before proving our main
theorem, let us demonstrate how Borell's theorem easily implies the
classical Pr\'{e}kopa-Leindler inequality. 

Consider the measure $\mu$ on $\R^{n+1} = \{(x, z): x\in \R^n, z\in
\R\}$ with density $d\mu(x,z) = e^{-z}dxdz$. Since this is a
log-concave density, by Borell's theorem $\mu$ is a log-concave
measure. On the other hand we have that  for any $\phi$,
$\mu(\epi(\phi)) = \int e^{-\phi}$, and also,
\[
\epi(\varphi\infc_\lambda \psi) =
(1-\lambda)\epi(\varphi)+\lambda \epi(\psi).
\]
Therefore, using the log-concavity of $\mu$ we have that
\begin{eqnarray*}
\int e^{-\varphi \infc_\lambda \psi} &=& 
\mu((1 - \lambda) \epi(\varphi) + \lambda \epi(\psi))\\ 
&\ge&
\mu( \epi(\varphi) )^{1 - \lambda} \mu ( \epi(\psi) )^\lambda = 
\left(\int e^{-\varphi} \right)^{1-\lambda}
\left(\int e^{-\psi} \right)^\lambda,
\end{eqnarray*}
which is the Pr\'{e}kopa-Leindler inequality \eqref{eq:PL}.

Our Theorem \ref{Thm-ginf-measurable} is a concavity property of the
measure $\mu$, with power $(-1)$, but with respect to a dif{}ferent
addition operation on the epi-graphs, the one that is induced by the
$\A$ transform. We shall make use of the following lemma regarding
the pullback of $\mu$ via $F$.

\begin{lem}\label{Lem-nu-conc}
Let $\nu$ be the measure on $\R^n\times \R^+$ given by
$d\nu (x,z) = e^{-1/z}z^{-(n+2)}dzdx$. Then $\nu $ is $(-1)$-concave
and for any measurable $\phi:\R^n \to \R^+$ we have 
\[
\int_{\R^n}e^{-\phi}=
\nu(F(\epi(\phi))).
\]	
\end{lem}

\begin{proof}
The dif{}ferential of $F$ is an upper triangular matrix, with
diagonal entries $1/z,\dots,1/z,-1/z^2$, thus $|\det (D_F(x,z))| = 
z^{-(n + 2)}$. It follows that the pullback of $\mu$ under $F$ has
density $e^{-1/z}z^{-(n+2)}$, and 
\[
\int_{\R^n}e^{-\phi}=\int_{\epi(\phi)}e^{-z}dzdx =
\mu\left(\epi(\phi)\right)=
\nu\left(F(\epi \phi) \right).
\]
It is easy to check that the density of $\nu$ is $-\frac{1}{n+2}$
concave. Theorem \ref{thm-Borell} thus implies that $\nu$ is
$(-1)$-concave. 
\end{proof}

\begin{proof}[{\bf Proof of Theorem \ref{Thm-ginf-measurable}}]
Let $\varphi, \psi, \eta:\R^n \to [0, \infty]$ be defined by
$f = e^{-\varphi},
 g = e^{-\psi},
 h = e^{-\eta}$. By Proposition \ref{Prop_new-formula} it follows
that $\varphi  \boxdot_\lambda  \psi    \ge     \eta$ so
$\epi(\varphi  \boxdot_\lambda  \psi) \subseteq \epi(\eta)$ i.e.
\[
(1-\lambda)F(\epi \varphi) +\lambda F(\epi \psi)
	=	  F(\epi(\varphi  \boxdot_\lambda    \psi))
\subseteq F(\epi\eta).
\]
By Lemma \ref{Lem-nu-conc} the measure $\nu$ is $(-1)$-concave and thus
\begin{eqnarray*}
	\mu (\epi\eta)&=& \nu(F(\epi\eta)) \ge
	M_{-1}^\lambda (\nu(F(\epi \varphi)), \nu (F(\epi \psi))) \\ &=&
	M_{-1}^\lambda (\mu (\epi \varphi), \mu(\epi \psi)).
\end{eqnarray*}
This means
\[
\int_{\R^n}e^{-\eta}\ge
\left(
(1-\lambda)\left(\int_{\R^n}e^{-\varphi}\right)^{-1} + \lambda\left(\int_{\R^n}e^{-\psi}\right)^{-1}
\right)^{-1},
\]
as required.
\end{proof}

Using the same methods, we give a proof of Theorem
\ref{Thm-Lp-norms}. To this end, we use the following lemma (the case
$n+1 < p$ first appeared in \cite{Rin}, Lemma 2).
\begin{lem}\label{Lem-nup-conc}
Let $p>0$ and let $\nu_p$ be the measure on $\R^n\times \R^+$ given by 
$d\nu_p (x,z) = pz^{p-(n+1)}dzdx$. Then for any measurable $\phi:\R^n \to \R^+$ we have 
	\[
	\int_{\R^n}\frac{1}{\phi^p}=
	\nu_p(F(\epi(\phi))).
	\]
If $p\in (0,n+1)$ then $\nu_p$ is log-concave, and if $n+1 \le p$
then $\nu_p$ is $\frac{1}{p}$-concave.
\end{lem}

\begin{proof}
Consider the measure $\mu_p$ on $\R^n \times \R^+$ given by  
\[
d\mu_p = \frac{p}{z^{p+1}}dxdz.
\]
The measure $\nu_p$ is the pullback of $\mu_p$ under the map $F$,
thus for a measurable function $\phi: \R^n \to \R^+$ we have
\[
\nu_p(F (\epi \phi)) =
\mu_p(\epi \phi ) =
\int_{\R^n} \int_{\phi(x)}^{\infty} \frac{p}{z^{p+1}} dz dx =
\int_{\R^n} \frac{1}{\phi(x)^p} dx.
\]
The density of $\nu_p$ is given by $p z^{p-(n+1)}$. If $p\in (0,n+1)$
the exponent ${p-(n+1)}$ is negative, and then the density is a
log-concave function, which implies that $\nu_p$ is log-concave. If
$p=n+1$ then the density is constant, and then $\nu_p$ is
proportional to the $(n+1)$-dimensional Lebesgue measure, which is
$\frac1{n+1}$-concave. If $n+1 < p$ the exponent ${p-(n+1)}$ is
positive, and then the density is
$\left( \frac{1}{p-(n+1)} \right)$-concave, which by Borell's Theorem
\ref{thm-Borell} implies that $\nu_p$ is $\frac{1}{p}$-concave.
\end{proof}

\begin{proof}[{\bf Proof of Theorem \ref{Thm-Lp-norms}}]
Given $f,g,h$ satisfying the conditions of the theorem, and $\lambda
\in(0,1)$, let $\varphi, \psi, \eta$ be defined by
$\varphi = \frac{1-\lambda}{f}, \psi = \frac{\lambda}{g}$, and
$\eta = \frac{1}{h}$. The assumption on $f,g,h$ implies that for
any $x, y\in \R^n$ and $t\in (0,1)$ we have
\[
\eta((1-t)x + ty) \le
\max \left\{\frac{1-t}{1-\lambda} \varphi(x),
\frac{t}{\lambda} \psi(y) \right\}.
\]
By Proposition \ref{Prop_new-formula} it follows that
$\eta \le \varphi \boxdot_\lambda \psi$, i.e. 
$\epi(\varphi \boxdot_\lambda \psi) \subseteq \epi(\eta)$, or
equivalently
\[
(1-\lambda)F(\epi \varphi) +\lambda F(\epi \psi) =
F(\epi(\varphi  \boxdot_\lambda    \psi)) \subseteq F(\epi \eta).
\]
If $n+1\le p$ then by Lemma \ref{Lem-nup-conc} the measure $\nu_p$ is
$\frac{1}{p}$-concave and thus
\begin{eqnarray*}
	\int_{\R^n} h^p
&=&
	\int_{\R^n} \frac{1}{\eta^p} =
	\nu_p(F(\epi(\eta))) \ge \\
&\ge&
	\nu_p \left( (1-\lambda)F(\epi \varphi) +\lambda F(\epi \psi) \right) \ge \\
&\ge&
	M_{1/p}^\lambda (\nu_p(F(\epi \varphi)), \nu_p (F(\epi \psi))) = \\
&=&
	M_{1/p}^\lambda
		\left(
			\int_{\R^n} \frac{1}{\varphi^p},\, \int_{\R^n} \frac{1}{\psi^p}
		\right) = \\
&=&
	\left(
		(1-\lambda)\left( \int_{\R^n} \frac{1}{\varphi^p} \right)^{\frac{1}{p}} + 
		\lambda \left( \int_{\R^n} \frac{1}{\psi^p} \right)^{\frac{1}{p}}
	\right)^p = \\
&=&
\left(
\left( \int_{\R^n} f^p \right)^{\frac{1}{p}} + 
\left( \int_{\R^n} g^p \right)^{\frac{1}{p}}
\right)^p,
\end{eqnarray*}
which completes the proof for the case $n+1 \le p$.
If $p\in (0, n+1)$ then by Lemma \ref{Lem-nup-conc} the measure
$\nu_p$ is log-concave and thus, choosing $\lambda =
\frac{\|g\|_p}{\|f\|_p + \|g\|_p}$ we get
\begin{eqnarray*}
	\int_{\R^n} h^p
&=&
	\int_{\R^n} \frac{1}{\eta^p} =
	\nu_p(F(\epi(\eta))) \ge \\
&\ge&
	\nu_p \left( (1-\lambda)F(\epi \varphi) +\lambda F(\epi \psi) \right) \ge \\
&\ge&
	(\nu_p(F(\epi \varphi)))^{1-\lambda}\cdot (\nu_p (F(\epi \psi)))^\lambda = \\
&=&
	\left( \int_{\R^n} \frac{1}{\varphi^p} \right)^{1-\lambda}\cdot
	\left( \int_{\R^n} \frac{1}{\psi^p}    \right)^\lambda = \\
&=&
	\left( \int_{\R^n} \left(\frac{f}{1-\lambda}\right)^p \right)^{1-\lambda}\cdot
	\left( \int_{\R^n} \left(\frac{g}{  \lambda}\right)^p    \right)^\lambda = \\
&=&
	\left(
		\|f\|_p + \|g\|_p
	\right)^p,
\end{eqnarray*}
which completes the proof for the case $p\in (0, n+1)$.
\end{proof}

\begin{rem}\label{rem-pguyimpliesBM}
As an example let us consider the $\lambda$-inf convolution of
\[ 
\varphi    = \max\left\{\|\cdot\|_K, 1\right\} \qquad {\rm and}
\qquad\psi = \max\left\{\|\cdot\|_T, 1\right\}.
\]
It may be checked (by either one of the definitions in Propositions
\ref{Prop_ginf-for-positive-is-well-def}, \ref{Prop_new-formula})
that
\[
\eta := \varphi \jinfc_\lambda \psi =
\max \left\{ \|\cdot\|_{(1-\lambda)K + \lambda T}, 1\right\}.
\]
Therefore, using the fact that for $p>n$ the following integrals
converge, and are simply multiples of the volumes of the
corresponding bodies, we get that the inequality
\[
\left( \int_{\R^n} \frac{1}{\eta^p} \right)^{\frac{1}{p}}
\ge
(1-\lambda)\left( \int_{\R^n} \frac{1}{\varphi^p} \right)^{\frac{1}{p}} + 
\lambda \left( \int_{\R^n} \frac{1}{\psi^p} \right)^{\frac{1}{p}}
\]
 implies 
\[
 \left( \vol((1-\lambda)K + \lambda T) \right)^{\frac{1}{p}}
 \ge
 (1-\lambda)\left( \vol(K) \right)^{\frac{1}{p}} + 
 \lambda \left( \vol(T) \right)^{\frac{1}{p}}.
\]
For $p\to n$, we have reproduced the Brunn-Minkowski inequality. 
\end{rem}

\section{Relation to Busemann's theorem}\label{Sec_Buse}

Busemann's convexity theorem states that given a centrally symmetric convex body
$K$, its intersection body, whose radial function is given
by $r(u) = \vol_{n-1}(K\cap u^\perp)$, is convex. The proof appears in
\cite{Bus}, see also \cite[Theorem 8.1.10]{Gardner}. In fact,
Busemann proves more; without assuming central symmetry, his proof deals
with the volume of intersections of $K$ with half-spaces. One possible way of
stating his theorem is as follows:
\begin{thm}[Busemann]\label{thm-Busemann}
Let $E$ be an $n-2$ dimensional subspace of $\R^n$. For every $u\in
E^{\perp}$, denote by $H_u$ the closed $n-1$ dimensional
half-space $E+\R^+u$. Let $x_0, x_1 \in E^{\perp}$, and
let $K_0, K_1$ be compact convex subsets of $H_{x_0}, H_{x_1}$
respectively. For $\lambda\in (0,1)$, let $x_{\lambda} = (1-\lambda)
x_0 + \lambda x_1$ and $K_{\lambda} = \conv(K_0, K_1) \cap
H_{x_{\lambda}}$. Then
\[
\frac{|x_{\lambda}|}{\vol(K_{\lambda})}
\le
(1-\lambda) \frac{|x_0|}{\vol(K_0)} + \lambda \frac{|x_1|}{\vol(K_1)}.
\]
\end{thm}

In \cite{Barthel-Franz} Barthel and Franz of{}fer a generalization of
Busemann's theorem where the convexity assumptions on the
bodies is relaxed. In \cite{KYZ} Kim, Yaskin, and Zvavitch give a
dif{}ferent extension of Busemann's theorem. They show that when
instead of volume, one considers some even log-concave measure, with
respect to which the hyperplane intersections of some centrally
symmetric convex body $K$ are measured, the same conclusion holds,
that is, the radial function $1/\mu(K\cap u^\perp)$ defines a norm
(here $\mu$ on $u^\perp$ is understood via the restriction of its
density function). Their argument is based on Ball's result on
convexity of certain bodies associated with log-concave functions,
see \cite{Ball}, and \cite[Chapter 10]{AGM} for a discussion of these
bodies and their important role in asymptotic convex geometry. These
bodies were generalized by Bobkov \cite{Bobkov} to measures with
weaker concavity assumptions. In \cite{CFPP} Busemann's theorem is
extended to this larger class of measures, and a very short and
elegant proof for the convexity of these bodies is given. They prove
\begin{thm}[Cordero, Fradelizi, Paouris, Pivovarov] \label{thm-Busemann-CFPP}
Let $\psi:\R^n\to\R^+$ be an even function which is
$\left( -\frac1n \right)$-concave, that is, it satisfies
\[
\psi^{-\frac1n}((1-\lambda)x+ \lambda y) \le
(1-\lambda)\psi^{-\frac1n}(x)+ \lambda \psi^{-\frac1n}(y).
\]
Then the function $\Phi$ defined by $\Phi(0) = 0$ and for $z\neq 0$
\[
\Phi(z) = \frac{|z|}{\int_{z^\perp}\psi(x)dx}
\]
is a norm.  
\end{thm}

These theorems are strongly related to our main theorems, as we shall
see below. However, in order to prove our main Theorem
\ref{Thm_ginf-mu} we shall need a version of Busemann's Theorem which
combines several of the above generalizations and seems not to have
appeared in the literature. Namely, we need a log-concave measure
rather than Lebesgue volume, we work with half-spaces rather than
even measures and centrally symmetric bodies, and we do not assume
any kind of convexity on the sets involved. We prove
\begin{thm}\label{thm-Busemann-combined}
Let $e^{-\psi}:\R^{n+2}\to\R^+$ be a log-concave density and let $E$
be an $n$-dimensional subspace of $\R^{n+2}$. For every $u\in
E^{\perp}$, denote by $H_u$ the closed $(n+1)$-dimensional half-space
$E+\R^+u$. Let $x_0, x_1 \in E^{\perp}$ be linearly independent and
let $x_{\lambda} = (1-\lambda)x_0 + \lambda x_1$ for some $\lambda\in
(0,1)$. Assume that $K_0, K_\lambda, K_1$ are (relatively) open and
connected subsets of $H_{x_0}, H_{x_\lambda}, H_{x_1}$ respectively.
If for any $t\in (0,1)$:
\[
\left( (1-t)K_0 + t K_1 \right)\cap H_{x_\lambda} \subseteq K_\lambda,
\]
then
\[
\frac{|x_{\lambda}|}{\int_{K_\lambda} e^{-\psi}}
\le
(1-\lambda) \frac{|x_0|}{\int_{K_0} e^{-\psi}} + 
\lambda \frac{|x_1|}{\int_{K_1} e^{-\psi}}.
\]
\end{thm}
We shall prove Theorem \ref{thm-Busemann-combined} in Section
\ref{Sec-5} and devote the rest of this section to show how it
implies Theorem \ref{Thm_ginf-mu}.
\begin{proof}[{\bf Proof of Theorem \ref{Thm_ginf-mu}.}]
Recall that the functions $f,g,h:\R^n\to [0,1]$ satisfy
\[
h((1 - t)x + ty)\ge
\min
\left\{
	f(x)^{\frac{1-t}{1-\lambda}},
	g(y)^{\frac{ t }{  \lambda}}
\right\},
\]
for any $t\in (0,1)$ and $x,y\in\R^n$. Recall that $\measure$ is a
log-concave measure on $\R^n$, and $\lambda\in(0,1)$. Define
$\varphi_0,\,\varphi_1,\,\varphi_\lambda$ by $f = e^{-\varphi_0},\,
g = e^{-\varphi_1},\, h = e^{-\varphi_\lambda}$. The assumption on
$f, g, h$ implies that
\begin{equation}\label{eq-ginf-condition}
\varphi_\lambda \le \varphi_0\boxdot_\lambda\varphi_1.
\end{equation}  
Fix some $0<s_0 < s_1$ and set $s_\lambda = (1-\lambda) s_0 + \lambda s_1$.
Denoting $x_i=(0,s_i,1)$ for $i = 0,\, \lambda,\, 1$, we identify the
epi-graphs of the three functions $\varphi_0,\, \varphi_\lambda,\,
\varphi_1$ with the sets $K_i \subseteq H_{i} :=
\R^n \times \left( \R^{+}\cdot x_i \right) \subset \R^{n+2}$
as follows
\[ K_i = \{ (x,s_i z,z): z>\varphi_i(x)\}\subseteq H_i.\]
Assume that the measure $\measure$ has a log-concave density
$e^{-\alpha}:\R^n \to \R$, and define the density
$\psi: \R^n\times \R \times\R^{+}$ to be
$\psi(x, s, z) = \alpha(x) + z$. Note that
\[
\int_{\R^n} e^{-\varphi_i}d\mu =
\frac{1}{\sqrt{1 + s_i^2}}
\int_{K_{i}} e^{-\psi},
\]
where the integration on the right hand side is with respect to the
$(n+1)$-dimensional Hausdorf{}f measure $H_{n+1}$. In this
terminology we need to show that
\[
\frac{|x_\lambda|}{\int_{K_\lambda}e^{-\psi}} \le
(1 - \lambda) \frac{|x_0|}{\int_{K_0}e^{-\psi}}  +
     \lambda  \frac{|x_1|}{\int_{K_1}e^{-\psi}},
\]
which is exactly the statement of Theorem \ref{thm-Busemann-combined}.
We are thus left with showing that the conditions of the theorem are
met. Clearly $\psi$ is convex so we must show that for any
$t\in (0,1)$
\begin{equation}\label{eq-show-this-inclusion}
\left( (1-t)K_0 + t K_1 \right)\cap H_{\lambda} \subseteq K_\lambda.
\end{equation}
To this end we define $\widetilde{F}: \R^n\times\R\times\R^+$ by
\[
\widetilde{F}(x,s,z) = \left(\frac{x}{z}, \frac{s}{z}, \frac{1}{z} \right). 
\]
The map $\widetilde{F}$ is an involution, which maps segments to
segments (it is a {\em fractional linear map}). It is closely related
to the map $F$ from Section \ref{sec:ginf-formulae}. Indeed,
\begin{equation}\label{eq-F-vs-F-tilde}
\widetilde{F}(K_i) =
\left\{
	\left(\frac{x}{z}, s_i, \frac{1}{z}\right): z > \varphi_i(x)
\right\} = 
\left\{
\left(x, s_i, z\right): (x, z)\in F(\epi \varphi_i)
\right\},
\end{equation}
and by Proposition \ref{Prop_ginf-for-positive-is-well-def} and the
inequality \eqref{eq-ginf-condition}, we have
\begin{equation}\label{eq-ginf-cond-by-F}
(1-\lambda)F(\epi(\varphi_0)) + \lambda F(\epi(\varphi_1)) =
F(\epi (\varphi_0\boxdot_\lambda\varphi_1)) \subseteq
F(\epi (\varphi_\lambda)).
\end{equation} 
Let $ H_i' = \widetilde{F}(H_i)$ and $A_i = \widetilde{F}(K_i)\subseteq
H_i',\,$ for $i=0, \lambda, 1$. The half-spaces $H_i'$ are parallel, and
$H_\lambda' = (1-\lambda)H_0' + \lambda H_1'$. We have
\begin{eqnarray*}
\bigcup_{k_0\in K_0,\, k_1\in K_1} [k_0, k_1]\cap H_\lambda &=&
\widetilde{F}
\left(
	\widetilde{F}\left(\bigcup_{k_0\in K_0,\, k_1\in K_1} [k_0, k_1]\right)
	\cap
	\widetilde{F}(H_\lambda)
\right)\\ &=&
\widetilde{F}
\left(
	\left(\bigcup_{k_0\in K_0,\, k_1\in K_1} \widetilde{F}([k_0, k_1])\right)
	\cap
	H_\lambda'
\right)\\ &=&
\widetilde{F}
\left(
	\left( \bigcup_{a_0\in A_0,\, a_1\in A_1} [a_0, a_1] \right)
	\cap
	H_\lambda'
\right)\\ &=&
\widetilde{F}
\left(
	\left(\bigcup_{0\le \beta \le 1} (1-\beta)A_0 + \beta A_1\right)
	\cap
	H_\lambda'
\right)\\ &=&
\widetilde{F}
	\left(  (1-\lambda) A_0 + \lambda A_1 \right)\\ &=&
\widetilde{F}
\left(
	(1-\lambda) \widetilde{F}(K_0) + \lambda \widetilde{F}(K_1)
\right)
\subseteq K_\lambda,
\end{eqnarray*}
where the last inclusion follows from \eqref{eq-F-vs-F-tilde} and
\eqref{eq-ginf-cond-by-F}. We have established the condition
\eqref{eq-show-this-inclusion} and thus the proof is complete.
\end{proof}

\section{Generalized Busemann Theorem}\label{Sec-5}
In this section we prove Theorem \ref{thm-Busemann-combined}.
Recall that we are given an $n$-dimensional subspace $E$ of $\R^{n+2}$
(which we assume without loss of generality is spanned by the first
$n$ coordinates i.e. $E = \R^n\times (0,0) \subset \RR^n \times \RR^2$),
and $K_0, K_\lambda, K_1$,  which are subsets of
$H_0 = H_{x_0},\, H_\lambda = H_{x_\lambda}$, and $H_1= H_{x_1}$
respectively, and $x_\lambda = (1-\lambda) x_0 + \lambda x_1$ with
$x_0$ and $x_1$ linearly independent. These satisfy that for any $t\in(0,1)$:
\begin{equation}\label{eq-busemann-gen-assumption}
\left( (1-t)K_0 + t K_1 \right)\cap H_{\lambda} \subseteq K_\lambda.
\end{equation}
Geometrically this means that the ``one step convex hull'' of $K_0$ and $K_1$
(obtained by taking all segments connecting the two sets), when intersected with
$H_\lambda$ (also known as ``harmonic linear combination'' in \cite{Gardner}), is contained in $K_\lambda$. Denote $u_i = \frac{x_i}{|x_i|}$, and
for $r>0$ set
\[
m_i(r) = \int_{K_i\cap (\R^n+ru_i)} e^{-\psi}dH_n,\qquad
\rho(u_i) = \int_0^\infty m_i(r)dr = \int_{K_i} e^{-\psi}dH_{n+1},
\]
where the integration is with respect to Hausdorf{}f measures. Our
aim is to show that 
\[
\frac{|x_\lambda|}{\rho(u_\lambda)} \le
(1-\lambda) \frac{|x_0|}{\rho(u_0)} +
\lambda     \frac{|x_1|}{\rho(u_1)}.
\]
To this end we define the percentile functions $p_0, p_1:[0,1]\to \R^+$ as follows.
For $\theta\in [0,1]$ we let
\[
\theta = \frac
{\int_0^{p_0(\theta)}m_0(r)dr}
{\int_0^{\infty}     m_0(r)dr} = \frac
{\int_0^{p_1(\theta)}m_1(r)dr}
{\int_0^{\infty}     m_1(r)dr}.
\]
Since $K_i$ are open and connected, $m_i$ are positive and continuous
on their support. Thus $p_i$ are well defined, and moreover, they are
dif{}ferentiable, so we may write
\[
\frac
{\int_0^{\infty}m_0(r)dr}
{m_0(p_0(\theta))}
=
p_0'(\theta), \qquad{\rm and} \qquad \frac
{\int_0^{\infty}m_1(r)dr}
{m_1(p_1(\theta))}
=
p_1'(\theta). 
\]
Define $p_\lambda: [0,1]\to \R^+$ by
$p_\lambda(\theta) u_{\lambda} =
(1 - \beta(\theta)) p_0 (\theta)u_0 +
     \beta(\theta)  p_1 (\theta) u_1$, where
\[
\beta(\theta) =
\frac
{\lambda p_0(\theta) |x_1|}
{ (1-\lambda) p_1(\theta) |x_0| + \lambda p_0(\theta) |x_1| }.
\]
One computes that
\begin{equation}\label{Eq_p_lambda}
\frac{p_\lambda(\theta)     }{|x_\lambda|} =
\frac{p_0(\theta)p_1(\theta)}{ (1-\lambda)p_1(\theta)|x_0| + \lambda p_0(\theta)|x_1| } =
 M_{-1}^\lambda\left(\frac{p_0(\theta)}{|x_0|}, \frac{p_1(\theta)}{|x_1|}\right).
\end{equation}
Dif{}ferentiation of (\ref{Eq_p_lambda}) with respect to $\theta$
(after taking reciprocals) gives:
\[
\frac{|x_\lambda|}{p_\lambda^2(\theta)} p_\lambda'(\theta) =
\frac{(1-\lambda) |x_0|}{p_0^2(\theta)}p_0'(\theta) +  \frac{\lambda |x_1|}{p_1^2(\theta)}p_1'(\theta). 
\]
Let us next show that 
\begin{equation}\label{ineq:ms}
m_\lambda(p_\lambda(\theta)) \ge m_0(p_0(\theta))^{1-\beta(\theta)} m_1(p_1(\theta))^{\beta(\theta)}.
\end{equation}
Indeed, by \eqref{eq-busemann-gen-assumption} we have in particular that 
\begin{equation}\label{Eq-some-inclusion}
\left((1 - \beta(\theta))  K_0 
+ \beta(\theta)   K_1\right) \cap H_\lambda
\subseteq K_\lambda.
\end{equation}
We may intersect this inclusion with the $(n+1)$-dimensional af{}fine
subspace $M$ containing $E_0 = \R^n + p_0(\theta)u_0$ and
$E_1 = \R^n + p_1(\theta)u_1$. Note that $M\cap H_i = E_i$, for
$i=0, \lambda, 1$, where we denoted
$E_\lambda = \R^n+p_\lambda(\theta)u_\lambda $ (and, in particular,
$K_i\cap M = K_i \cap E_i$). Intersecting with $M$, the inclusion
\eqref{Eq-some-inclusion} implies
\begin{eqnarray*}
(1-\beta(\theta)) \left( K_0\cap M\right)
+ \beta(\theta)  \left( K_1\cap M \right)&\subseteq&\left((1 - \beta(\theta))  K_0 
+ \beta(\theta)   K_1\right) \cap M \cap H_\lambda
\\&\subseteq&
K_\lambda\cap M.
\end{eqnarray*}
Finally, we use the log-concavity of the density $e^{-\psi}$ (on $M$)
together with Pr\'ekopa's theorem \cite{Prekopa} on the log-concavity
of the marginal of a log-concave density (which follows from
\eqref{eq:PL}, for example), the density here being $e^{-\psi}$
restricted to $\conv (K_0\cap M, K_1\cap M)$. We get that 
\[
\int_{K_\lambda \cap E_\lambda} e^{-\psi}
\ge
\left(\int_{K_0 \cap E_0} e^{-\psi}\right)^{1-\beta(\theta)}
\cdot
\left(\int_{K_1 \cap E_1} e^{-\psi}\right)^{\beta(\theta)},
\]
which is \eqref{ineq:ms}. The rest of the argument follows closely
the classical Busemann argument.
\begin{eqnarray*}
\frac{\rho(u_\lambda)}{|x_\lambda|} &=& \frac
{\int_{K_\lambda}e^{-\psi}}
{|x_\lambda|}=
\frac
{\int_0^\infty m_\lambda(r)dr}
{|x_\lambda|}
= \qquad\qquad\qquad\\ \\ 
&=&
\int_0^1
\frac
{m_\lambda(p_\lambda(\theta))p_\lambda'(\theta)} 
{|x_\lambda|}d\theta = 
\int_0^1 
\frac
{m_\lambda(p_\lambda(\theta))p_\lambda^2(\theta)}
{|x_\lambda|^2}
M_{1}^\lambda \left(\frac{|x_0| p_0'(\theta)}{p_0^2(\theta)}, \frac{|x_1| p_1'(\theta)}{p_1^2(\theta)}\right)d\theta \\ \\
&=&
\int_0^1 
\frac
{m_\lambda(p_\lambda(\theta))p_\lambda^2(\theta)}
{|x_\lambda|^2}
M_{1}^\lambda \left(\frac{|x_0| \rho(u_0)}{m_0(p_0(\theta))p_0^2(\theta)}, \frac{|x_1| \rho(u_1)}{m_1(p_1(\theta))p_1^2(\theta)}\right)d\theta \\ \\
&=&
\int_0^1 
m_\lambda(p_\lambda(\theta))
M_{-1}^\lambda\left(\frac{p_0(\theta)}{|x_0|}, \frac{p_1(\theta)}{|x_1|}\right)^2 
M_{1}^\lambda \left(\frac{|x_0| \rho(u_0)}{m_0(p_0(\theta))p_0^2(\theta)}, \frac{|x_1| \rho(u_1)}{m_1(p_1(\theta))p_1^2(\theta)}\right)d\theta.
\end{eqnarray*}
Denoting $w_i = w_i(\theta) := \frac{|x_i|}{p_i(\theta)}$ and
$a_i = a_i (\theta) :=
\frac{w_i(\theta) \rho(u_i)}{|x_i| m_i(p_i(\theta))}$ we get,
using \eqref{ineq:ms}
\begin{eqnarray*}
\frac{\rho(u_\lambda)}{|x_\lambda|} &\ge&
\int_0^1 \frac{m_0(p_0(\theta))^{1-\beta(\theta)} m_1(p_1(\theta))^{\beta(\theta)}}
{((1 - \lambda) w_0(\theta) + \lambda w_1(\theta))^2}
M_1^\lambda 
\left(  \frac{\rho(u_0)w_0^2(\theta)}{|x_0| m_0(p_0(\theta))}, \frac{\rho(u_1)w_1^2(\theta)}{|x_1| m_1(p_1(\theta))}  \right) d\theta \\ \\
&=&
\int_0^1
\frac
{m_0(p_0)^{\frac{(1 - \lambda) w_0}{(1 - \lambda) w_0 + \lambda w_1}} m_1(p_1)^{\frac{\lambda w_1}{(1 - \lambda) w_0 + \lambda w_1}}}
{(1 - \lambda) w_0 + \lambda w_1}
\left( \frac{(1 - \lambda) w_0}{(1 - \lambda) w_0 + \lambda w_1}a_0 +  \frac{\lambda w_1}{(1 - \lambda) w_0 + \lambda w_1}a_1 \right)d\theta.
\end{eqnarray*}
Applying the arithmetic-geometric means inequality we get 
\begin{eqnarray*}
\frac{\rho(u_\lambda)}{|x_\lambda|} &\ge&
\int_0^1
\frac
{(a_0m_0(p_0))^{\frac{(1 - \lambda) w_0}{(1 - \lambda) w_0 + \lambda w_1}} (a_1 m_1(p_1))^{\frac{\lambda w_1}{(1 - \lambda) w_0 + \lambda w_1}}}
{(1 - \lambda) w_0 + \lambda w_1}d\theta \\ \\
&=&
\int_0^1
\frac
{(w_0 \rho(u_0)/|x_0|)^{\frac{(1 - \lambda) w_0}{(1 - \lambda) w_0 + \lambda w_1}}
((w_1 \rho(u_1)/|x_1|))^{\frac{\lambda w_1}{(1 - \lambda) w_0 + \lambda w_1}}}
{(1 - \lambda) w_0 + \lambda w_1}d\theta.
\end{eqnarray*}
Applying the geometric-harmonic means inequality in the last integral we get 
\[
\frac{\rho(u_\lambda)}{|x_\lambda|} \ge \int_0^1 M_{-1}^\lambda \left(\frac{\rho(u_0)}{|x_0|}, \frac{\rho(u_1)}{|x_1|}\right) d\theta = 
M_{-1}^\lambda \left(\frac{\rho(u_0)}{|x_0|}, \frac{\rho(u_1)}{|x_1|}\right)
\]
That is,
\[
\frac{|x_\lambda|}{\rho(u_\lambda)} \leq (1-\lambda) \frac{|x_0|}{\rho(u_0)} + \lambda \frac{|x_1|}{\rho(u_1)},
\]
as required. \qed

\smallskip \noindent
Shiri Artstein-Avidan \\
School of Mathematical Sciences \\
Tel Aviv University, Tel Aviv 69978, Israel  \\
{\it e-mail}: artstein@post.tau.ac.il \\

\smallskip \noindent
Dan Florentin \\
Department of Mathematical Sciences  \\
Kent State University, Kent, OH, 44242, USA  \\
{\it e-mail}: danflorentin@gmail.com  \\

\smallskip \noindent
Alexander Segal \\
 Afeka Academic College of Engineering, Tel Aviv 69107, Israel\\
{\it e-mail}: segalale@gmail.com  \\

\end{document}